\documentclass[12pt]{amsart}

\usepackage{cancel}
\usepackage{latexsym, amssymb, amsmath}
\usepackage{soul}
\usepackage{amsfonts, graphicx}
\usepackage{graphicx,color}
\newcommand{\be}{\begin{equation}}
\newcommand{\ee}{\end{equation}}
\newcommand{\beq}{\begin{eqnarray}}
\newcommand{\eeq}{\end{eqnarray}}

\newtheorem{thm}{Theorem}[section]

\newtheorem{lma}{Lemma}[section]

\newtheorem{cor}{Corollary}[section]

\theoremstyle{remark}

\numberwithin{equation}{section}

\def\be{\begin{equation}}
\def\ee{\end{equation}}
\def\bee{\begin{equation*}}
\def\eee{\end{equation*}}

\def\K{K\"ahler }

\def\KR{K\"ahler-Ricci }

\begin{document}

\title[]
{On the uniqueness of Ricci flow}

 \author{Man-Chun Lee}
\address[Man-Chun Lee]{Department of Mathematics 
University of British Columbia 
121-1984 Mathematics Road 
Vancouver, B.C. V6T 1Z2, Canada}
\email{mclee@math.ubc.ca}


\date{\today}
\begin{abstract}
In this note, we study the problem of the uniqueness of solutions to the Ricci flow on complete noncompact manifolds. We consider the class of solutions with curvature bounded above by $C/t$ when $t>0$ and prove a uniqueness result when initial curvature is of polynomial growth and Ricci curvature of the flow is relatively small.
\end{abstract}


\maketitle

\markboth{Man-Chun Lee}{On the uniqueness of Ricci flow}
\section{introduction}
Let $(M^n,g_0)$ be a complete Riemannian manifold. In \cite{Ham1982}, Hamilton introduced the Ricci flow
$$\frac{\partial}{\partial t} g=-2Ric(g(t)),\;\;g(0)=g_0,$$
and established the short-time existence and uniqueness on compact manifolds. Later on, using the idea of DeTurck's trick \cite{Deturck}, Shi \cite{Shi} and Chen-Zhu \cite{ChenZhu2006} generalized the existence and uniqueness result to complete noncompact manifolds with bounded curvature. They asserted that for any complete noncompact manifold $(M,g_0)$ with bounded curvature, there is a unique short-time solution of the Ricci flow starting from $g_0$ which has bounded curvature. It is natural to ask to what extent we have the short-time existence and uniqueness of the Ricci flow on a general complete noncompact manifold.

When $n=2$, Giesen and Topping \cite{GT1,top} successfully extended the classical results. In particular, they showed that any initial surface (including those that are incomplete and with unbounded curvature) can be flowed in a unique way by a smooth and instantaneously complete solution and the maximal existence time can be explicitly calculated. In case of $n=2$, the Ricci flow can be reduced to a logarithmic fast diffusion equation whose study can be reduced to a single scalar equation. However for $n\geq 3$, the Ricci flow is a system of nonlinear weakly parabolic partial differential equation. It is unclear how much of their work can be extended to higher dimension. 

There are nevertheless a number of existence results in which the initial metric is complete with possibly unbounded curvature. In \cite{Simon}, Simon proved that starting from any sufficiently small $C^0$ perturbation $g_0$ of a complete Riemannian metric with bounded curvature, there is a short-time solution of the Ricci harmonic heat flow. We also refer to the work \cite{Kol,SSS1} where the Ricci harmonic heat flow is solved starting with rough initial data obtained from a sufficiently small perturbation of the Euclidean metric on $\mathbb{R}^n$ and \cite{SSS2} for a similar result in hyperbolic space. In \cite{xu}, Xu constructed a solution from a metric $g_0$ which satisfies a Sobolev inequality and the curvature is bounded in some $L^p$ sense. Fei \cite{He2} constructed a solution when the initial metric is locally Euclidean in appropriate sense using the pseudolocality of the Ricci flow. When $M=\mathbb{C}^n$, Chau-Li-Tam \cite{ChauLiTam} and Yang-Zheng \cite{YangZheng} constructed solutions to the \KR flow which are $U(n)$-invariant. 

In \cite{Cabezas-RivasWilking2011}, Cabezas-Rivas and Wilking obtained a Ricci flow solution when the initial metric $g_0$ has non-negative complex sectional curvature. Moreover the curvature of the solution $g(t)$ will be bounded by $c/t$ for some constant $c>0$ if in addition the initial metric is non-collapsing in the sense that the volume of every geodesic ball of radius 1 is bounded below by a fixed positive constant $v_0$. In the non-collapsing case, the result has been recently generalized to various situations. When the initial metric is \K with non-negative bisectional curvature, the author and Tam \cite{LeeTam2017} were able to construct a short-time solution to the \KR flow. When the initial metric is only Riemannian and has almost weakly $\text{PIC}_1$, Yi \cite{Yi2018} showed that a short-time solution also exists. We also refer the earlier work by Bamler, Cabezas-Rivas and Wilking \cite{BamlerCabezasWilking2017} for the case when $g_0$ has complex sectional curvature bounded from below or equivalently $g_0$ has almost weakly $\text{PIC}_2$. On the other hand, Simon and Topping \cite{SimonTopping2017} used the ideas of Hochard \cite{Hochard2016} and together with their results in \cite{SimonTopping2016} to prove that similar results are true for three dimensional Riemannian manifolds under a weaker condition that the Ricci curvature is bounded from below. 

Remarkably, although the initial metric does not have uniformly bounded curvature on $M$, it is possible that the curvature of the evolving metric $g(t)$ becomes bounded instantaneously. Moreover, most of the examples mentioned above satisfy $|Rm(g(t))|\leq Ct^{-1}$ for some $C>0$ when $t$ is small and positive.

Comparatively, not much is known about the uniqueness except for some special cases. To the best of our knowledge, it is still unknown whether any of the above examples are unique. There are a very few results in this direction. For example, Chau-Li-Tam \cite{ChauTamLi} (see also \cite{fan}) showed the uniqueness of K\"ahler-Ricci flow when the flows are uniformly equivalent to a fixed metric with bounded curvature. Sheng-Wang \cite{SW} studied complete solutions with lower bound on complex sectional curvature and proved uniqueness under some extra assumptions on the initial metric. When $n=3$, Chen \cite{chen} obtained a strong uniqueness result which says that any complete solution of the Ricci flow starting from the Euclidean metric must be stationary. Recently, Kotschwar \cite{KB,KB2} introduced an energy method which extended the classical uniqueness result to the case when two flows are uniformly equivalent and their curvature is bounded above by $C(d_0(x,p)^2+1)/t^\gamma,\gamma<1/2$ or $C/t^\gamma,\gamma<1$. However, most of the solutions mentioned in the constructions above satisfy a curvature bound of the form $C/t$. 

In this note, we discuss the uniqueness problem on the class of solutions with curvature bound $C/t$. Within the \K category, we consider flows which are uniformly equivalent. In this work, we will prove the following.

\begin{thm}\label{kahler}
Let $(M^n,g_0)$ be a complete noncompact \K manifold with complex dimension $n$. Suppose $g(t)$  and $\tilde g(t)$ is a solution to the \K Ricci flow with initial data $g_0$. Assume  further that $\exists \lambda>1,\, C_1>0$ such that 
\begin{enumerate}
\item[(i)] $\lambda^{-1} \tilde g(t)\leq g(t)\leq \lambda \tilde g(t)\quad\text{on}\;\;M \times [0,T]$ and
\item[(ii)] 
$$|Rm(g(t))|\leq \frac{C_1}{t}\quad\text{on}\;\;M\times (0,T].$$
\end{enumerate}
Then $g(t)=\tilde g(t)$ on $[0,T]$.
\end{thm}
In fact, one might only require $\tilde g(t)$ to be \K by the result of Huang and Tam \cite{HuangTam2018}. In particular, they showed that the Ricci flow will remain \K if it is \K initially and satisfies $|Rm|(g(t))\leq Ct^{-1}$ for some $C>0$. In the most of the recent constructions discussed above, the solutions satisfy non-integrable curvature bounds, and it is hard to say whether two such solutions remain uniformly equivalent to each other. From this point of view, we discuss Ricci flows in which uniform equivalence is not assumed. We use the energy method in \cite{KB} to prove that if the curvature of the initial metric is of polynomial growth and the curvature of flows is bounded by $a/t$ where $a$ is relatively small, then the solutions must agree. 

\begin{thm}\label{real1}
For any $m\in\mathbb{N}$, $\exists\, \epsilon=\epsilon(m)>0$ such that the following holds: Suppose $(M,g_0)$ is a complete noncompact manifold satisfying
$$|Rm(g_0)|\leq C_0( d(x,p)+1)^m$$
for some $C_0>0$ and a fixed point $p\in M$. If $g(t)$ and $\tilde g(t)$ are two smooth solutions to Ricci flow on $M\times [0,T]$ with $g(0)=\tilde g(0)=g_0$ for which 
$$|Rm(\tilde g(t))|_{\tilde g(t)}+|Rm(g(t))|_{g(t)}\leq \frac{\epsilon}{t}$$
on $M\times (0,T]$, then $g(t)=\tilde g(t)$ for all $t\in [0,T]$.
\end{thm}

{\it Acknowledgement}: The author would like to thank his advisor Professor Luen-Fai Tam for his constant support and encouragement. He also thanks the referee for useful comments.

\section{Uniqueness of \KR flow}
In \cite{KB2}, the author consider the uniqueness problem if the curvature is bounded above by $Ct^{-\gamma}$ for $\gamma\in [0,1)$. By integrability of the Ricci curvature, two solutions $g(t)$ and $\tilde g(t)$ are also uniform equivalent to $g_0$ and hence to each other. In this section, we consider the problem within \K category and extend the uniqueness to the case of $\gamma=1$ provided that the $g(t)$ and $\tilde g(t)$ are uniformly equivalent along the flow.

\begin{proof}[Proof of Theorem \ref{kahler}]
By parabolic re-scaling, we may assume $T=1$. Since $\bar g=g(1)$ has bounded curvature. By \cite[Theorem 1]{TamExh}, $\exists\, \rho\in C^\infty(M)$ such that 
\be\label{e-exhaution-1}
 \left\{
   \begin{array}{ll}
     &1+d_{\bar g}(x,p)\le \rho(x)\le C_2(1+d_{\bar g}(x,p)) \\
       & |\nabla_{\bar g}\rho|_{\bar g}\le C_2  \\
       & |\nabla^2_{\bar g}\,\rho|\le C_2
   \end{array}
 \right.
\ee
for some $C_2>0$. By Ricci flow equation, for all $t\in (0,1]$,
\begin{align}\label{eq1}
\lambda^{-1}t^\alpha  \bar g\leq \tilde g(t),\,g(t)\leq \frac{\lambda}{t^\alpha} \bar g
\end{align}
where $\alpha=2(n-1)C_1$. We may assume $\alpha=1+\beta>1$

As in \cite{ChauTamLi}, we define $w:M\times [0,1]\rightarrow \mathbb{R}$ to be
$\displaystyle w(x,t)=\int^t_0 \log \frac{\det \,g(x,s)}{\det\, \tilde g(x,s)}\,ds$. Let $h(x,t,s)=sg(x,t)+(1-s)\tilde g(x,t)$, then
\begin{align*}
\frac{\partial w}{\partial t}&=\log\frac{\det\,g(x,t)}{\det \, \tilde g(x,t)}\\
&=\int^1_0 \frac{\partial}{\partial s}\log\det\, h(x,t,s)\,ds\\
&=\left(g_{i\bar j}(x,t)-\tilde g_{i\bar j}(x,t)\right)\int^1_0 h^{i\bar j}(x,t,s)\,ds.
\end{align*}
On the other hand,
\begin{align*}
w_{i\bar j}(x,t)&=\int^t_0 -R_{i\bar j}(x,s)+\tilde R_{i\bar j}(x,s) \,ds\\
&=g_{i\bar j}(x,t)-\tilde g_{i\bar j}(x,t).
\end{align*}
Hence,
$$\frac{\partial }{\partial t}w(x,t)= w_{i\bar j}(x,t)\int^1_0 h^{i\bar j}(x,t,s) \,ds.$$
In particular,
\begin{align*}
&\quad \frac{\partial }{\partial t}w^2(x,t)- (w^2)_{i\bar j}(x,t)\int^1_0 h^{i\bar j}(x,t,s)ds \\
&=2w \left(\partial_t w(x,t) -w_{i\bar j}(x,t)\int^1_0 h^{i\bar j}(x,t,s)ds\right)-2w_i(x,t) w_j(x,t)\int^1_0 h^{i\bar j}(x,t,s)ds\\
&\leq 0.
\end{align*}
Let $\eta(x,t):M\times (0,1]\rightarrow [0,+\infty)$ by $\eta(x,t)=f(t)\rho(x)$ where $f(t)=\exp\left(-\frac{2\lambda C_2}{\beta t^{\beta}}\right)$. Using (\ref{e-exhaution-1}) and (\ref{eq1}), we have
\begin{align*}
 &\quad \frac{\partial }{\partial t}\eta(x,t)-\eta_{i\bar j}(x,t)\int^1_0 h^{i\bar j}(x,t,s) \,ds\\
&=f'(t)\rho(x)-f(t)\rho_{i\bar j}(x,t)\int^1_0 h^{i\bar j}(x,t,s)\, ds\\
&\geq f'(t)\rho(x)-\frac{\lambda}{t^{1+\beta}} f(t)|\bar\nabla^2 \rho|_{\bar g} \\
&\geq f'(t)-\frac{\lambda C_2}{t^{1+\beta}} f(t)\\
&>0.
\end{align*}
For any $\epsilon>0$, the function $F_\epsilon(x,t)=w^2(x,t)-\epsilon \eta(x,t)- \epsilon t$ satisfies 
\begin{align}\label{max}
\frac{\partial }{\partial t}F(x,t)-F_{i\bar j}(x,t)\int^1_0 h^{i\bar j} (x,t,s)\,ds<0\quad\text{on}\;\;M\times (0,1].
\end{align}
Since $g(t)$ and $\tilde g(t)$ are uniformly equivalent, $w^2\leq Ct^2$ for some $C>0$. Hence, $\exists\, t_0\in (0,1]$ such that for all $t\in [0,t_0]$, $w^2-\epsilon t<0$. On the other hand, as $w^2$ is bounded, there exists a compact set $K$ such that for all $(x,t) \in M\setminus K\times (t_0,1]$, 
$$F_\epsilon(x,t)<0.$$
Suppose $F_\epsilon>0$ somewhere, then it achieves its positive maximum at $(x',t')\in K\times (t_0,1]$ which is impossible by (\ref{max}). Therefore, for any $\epsilon>0$, $F_\epsilon(x,t)\leq 0$ on $M\times [0,1]$. By letting $\epsilon\rightarrow 0$, we conclude that $w\equiv 0$ which implies $g(t)=\tilde g(t)$ on $M\times[0,1]$ by differentiating $w(x,t)$ with respect to space.
\end{proof}

\section{Uniqueness of Ricci flow }
If the curvature is bounded above by $Ct^{-\gamma}$ where $\gamma<1$, then one can easily see  that two flows are equivalent along the flow. However there are examples of solutions with curvature bounded above by $Ct^{-1}$ which is not integrable, then it is hard to say if two Ricci flow are uniformly equivalent.
In this section, we consider the uniqueness problem of the Ricci flow without the assuming the uniform equivalence of the solutions. Instead, we consider the Ricci flow solutions starting from a metric with curvature of at most polynomial growth bound and enjoy $a/t$ where $a$ is small relative to the polynomial order. More precisely, we prove the following.

\begin{thm}\label{real}
For any $m\in\mathbb{N}$, $\exists\, \epsilon=\epsilon(m)>0$ such that the following holds: Suppose $(M,g_0)$ is a complete noncompact manifold satisfying
$$|Rm(g_0)|\leq C_0( d(x,p)+1)^m$$
for some $C_0>0$ and a fixed point $p\in M$. If $g(t)$ and $\tilde g(t)$ are two smooth solutions to the Ricci flow on $M\times [0,T]$ with same initial data $g(0)=\tilde g(0)=g_0$ and satisfy
$$|Rm(\tilde g(t))|_{\tilde g(t)}+|Rm(g)|_{g(t)}\leq \frac{K}{t}$$
together with  $$|Ric(\tilde g(t))|_{\tilde g(t)}+|Ric(g(t))|_{g(t)}\leq \frac{\epsilon}{t}$$
on $M\times (0,T]$ for some $K>0$, then $g(t)=\tilde g(t)$ for all $t\in [0,T]$.
\end{thm}

\subsection{Exhaustion function on $M$}
In order to apply maximum principle using energy method. We first  construct a exhaustion function on $M$.

\begin{lma}\label{exh}
Suppose $g(t)$ is a solution of Ricci flow on $M\times [0,T]$ satisfying 
$$|Ric(g(t))|\leq \frac{1}{4t}\quad\text{on}\;\;M\times (0,T].$$
Then for any $L_1,L_2>0$, there exists $\tau=\tau(L_1,L_2,T)>0$ and a smooth function $\eta:M \times [0,\tau]\rightarrow [0,+\infty)$ such that 
$$\frac{\partial \eta}{\partial t}\geq L_1 |\nabla \eta|^2\quad\text{and}\quad \eta(x,t)\geq L_2 r(x)^2$$
on $M\times (0,T]$ where $r(x)= d_{g(T)}(x,p)$ for some $p\in M$.
\end{lma}
\begin{proof}
By \cite[proposition 2.1]{GW}, there exists smooth function $\rho:M \rightarrow [0,+\infty)$ such that $d_{g(T)}(x,p)\leq \rho(x)\leq d_{g(T)}(x,p)+1$ and $| ^{g(T)}\nabla \rho|_{g(T)}\leq 2$. Due to the curvature assumption, for any $t\in (0,T]$,
\begin{align}\label{zerometric}\sqrt{\frac{t}{T}}\,g(T)\leq g(t)\leq \sqrt{\frac{T}{t}} \, g(T).
\end{align}

Let $\eta(x,t)= f(t) \rho^2(x)$, then $\frac{\partial \eta}{\partial t}=f'(t)\rho^2 $ while
$$|\nabla \eta|^2=f^2(t)|\nabla \rho|^2= f^2 g^{ij} \partial_i \rho \,\partial_j \rho\leq 2\sqrt{\frac{T}{t}} f^2.$$
Hence if we take 
$$f(t)=\frac{1}{a-4L_1\sqrt{T}\sqrt{t} }$$
then the first inequality holds. The second inequality holds when we choose $a$ and $\tau$ small enough.
\end{proof}

\subsection{Estimates on curvature and its derivatives}
In this subsection, we use a result in \cite{chen} to modify estimate of $|Rm(g(t))|$ in term of $d_t(x,p)$ if $|Rm(g(t))|<at^{-1}$ and initial curvature is of polynomial growth. 

\begin{lma}
Suppose $(M,g(t))$ is a complete solution of the Ricci flow on $M\times [0,T]$ which satisfies
$$|Rm|_{g(0)}\leq C_1 (d_0^m(x,p)+1)\quad\text{and}\quad |Rm|_{g(t)}\leq \frac{a}{t}\quad\text{on}\;\;M\times (0,T]$$
for some $a,m,C_1>0$, $p\in M$. Then there exists $C_2=C_2(n,m,C_1,a)>0$ such that
$$|Rm|(x,t)\leq C_2 (d_t(x,p)^m+1).$$
\end{lma}
\begin{proof}
Let $L>0$ be a fixed constant first. Let $x\in M$ so that $d_0(x,p)=R \geq L$. On $B_0(p,2R)$,
$$|Rm|_{g(0)}\leq C_1(2^m R^m+1)=r^{-2}.$$
We may assume $L=L(m,C_1)$ to be sufficiently large so that $r \leq R$. Then by \cite[Corollary 3.2]{chen},
$$|Rm|(x,t)\leq e^{C_n a}r^{-2}=C_1e^{C_n a}(2^m R^m+1),\;\quad\forall t\in [0,T].$$
But since $Ric\leq a(n-1)t^{-1}$ on $M\times (0,T]$, by \cite[Theorem 17.2]{Ham1995} (see also \cite[Lemma 8.3]{P}),
$$d_t(x,p) \geq d_0(x,p)-\beta_n \sqrt{at}.$$
Hence if we choose $L=L(a,n,C_1,m)$ even larger, then $2d_t(x,p)\geq d_0(x,p)$. Hence, $\exists C_2=C_2(n,a,C_1,m)>0$ such that  
$$|Rm|(x,t)\leq C_2 (d_t(x,p)^m+1)\quad\text{whenever}\; d_0(x,p)\geq L.$$
On the other hand, the curvature is bounded on $B_0(p,L)$. Use Chen's result again, we get the upper bound of $|Rm|_{g(t)}$ on $B_0(p,L)$. So by choosing a larger $C_2$, we conclude that
$$|Rm|(x,t)\leq C_2(d_t(x,p)^m+1)\quad\forall (x,t)\in M\times [0,T].$$
\end{proof}

By Shi-estimate, we can also obtain a estimate of $|\nabla Rm|$ in term of spatial information.

\begin{lma}\label{3.3}
Suppose $(M,g(t))$ is a complete solution to the Ricci flow $M\times  [0,T]$ which satisfies 
$$|Rm|_{g(0)}\leq C_1 (d_0^m(x,p)+1)\quad\text{and}\quad |Rm|_{g(t)}\leq \frac{a}{t}\quad\text{on}\;\;M\times (0,T]$$
for some $a,m,C_1>0$, $p\in M$. Then there exists $C_3=C_3(T,n,m,a,C_1)>0$ such that 
$$|\nabla Rm|_{g(t)}\leq \frac{C_3 [1+d_t(x,p)^m]^{3/2}}{\sqrt{t}}\quad \text{on}\;\;M \times (0,T].$$
\end{lma}
\begin{proof}
Since $Ric \leq a(n-1)t^{-1}$ on $(0,T]$, we conclude by integrating time that for $s\in [t/2,t]$,
$$\lambda^{-1}d_t(x,p)\leq d_{s}(x,p)\leq \lambda d_t(x,p)\quad\text{where}\;\;\lambda=2^{a(n-1)}.$$

Now we fix $t_0>0$, take $\epsilon=t_0/2$ and $R\geq 1$. For $t\in [\epsilon,2\epsilon]$, $B_\epsilon (p, 2 \lambda R)\subset B_{t}(p,2\lambda^2 R)$. Thus, on $B_\epsilon (p,2R)\times [\epsilon,2\epsilon]$,
\begin{align}\label{CC}|Rm|(z,t)\leq C_2[1+(2\lambda^2 R)^m]\leq K=C_3(R^m+1)\end{align}

Using \eqref{CC} with Shi's first order estimate \cite[Theorem 5.8]{HRF}, one can conclude that on $B_\epsilon (p,\lambda R)\times (\epsilon,2\epsilon]$,
\begin{align}\label{11}
|\nabla Rm|(z,t)\leq C_nK\left(\frac{1}{\lambda^2R^2}+\frac{1}{t-\epsilon}+K \right)^{1/2}.
\end{align}
Put $t=2\epsilon=t_0$ into \eqref{11},
\begin{equation}\begin{split}
|\nabla Rm|(z,2\epsilon)&\leq C(R^m+1)\left[\frac{1}{\lambda^2R^2}+\frac{1}{\epsilon}+(R^m+1)\right]^{1/2}\\
&\leq \frac{C_3 (R^m+1)^{3/2}}{t_0^{1/2}}.
\end{split}
\end{equation}
where $C_3=C_3(n,m,a,C_1,T)>0$. As $t_0>0$ is arbitrary, we know that if $R\geq 1$, then
\begin{align}|\nabla Rm|(x,t)\leq \frac{C_3 (R^m+1)^{3/2}}{\sqrt{t}}\quad \text{on}\;\;B_t(p,R).\end{align}
Now we wish to get pointwise estimate. If $d_t(x,p)\leq 1$, then
\begin{align*}
|\nabla Rm|(x,t) \leq \frac{4C_3}{\sqrt{t}}.
\end{align*}
If $d_t(x,p)=R\geq 1$, $x\in B_t(p,2R)$. And hence,
\begin{align*}
|\nabla Rm|(x,t) \leq \frac{C_3 [(2R)^m+1]^{3/2}}{\sqrt{t}}\leq \frac{C_4(1+d_t(x,p)^m)^{3/2}}{\sqrt{t}}.
\end{align*}
\end{proof}
\noindent
Combines the above result and interpolates with Shi-estimate \cite[Theorem 5.8]{HRF} in its standard form, we have the following estimate.
\begin{cor}\label{esti}
Suppose $(M,g(t))$ is a complete Ricci flow for $t\in [0,T]$ satisfying 
\begin{align}|Rm|_{g(0)}\leq C_1 (d_0^m(x,p)+1)\quad\text{and}\quad |Rm|_{g(t)}\leq \frac{a}{t}\quad\text{on}\;\;M\times (0,T]\end{align}
for some $a,m,C_1>0$, $p\in M$. Then there exists $C_3=C_3(T,n,m,a,C_1)>0$ such that for any $x\in M$, $t\in(0,T]$, $\delta\in [0,1]$,
\begin{align}|Rm(g(t))|\leq C_3\left(d_T(x,p)+1\right)^{m\delta} \left(\frac{1}{t}\right)^{1-\delta}
\end{align}
and 
\begin{align}|\nabla Rm(g(t))|\leq C_3\frac{(d_T(x,p)+1)^{3\delta m/2}}{t^{3/2-\delta}}.\end{align}
\end{cor}
\begin{proof}
The proof on the estimates of $|Rm(g(t))|$ and $|\nabla Rm(g(t))|$ are same. We only work on $|Rm(g(t))|$ here. When $\delta=0$, it was already proved in \cite[Theorem 5.8]{HRF}. By interpolating it with Lemma \ref{3.3}, we have for any $x\in M$, $t\in (0,T]$, $\delta\in [0,1]$,
\begin{align}\label{000}|Rm(g(t))|\leq C\Big(d_t(x,p)+1\Big)^{m\delta}\left( \frac{1}{t}\right)^{1-\delta}.\end{align}
It suffices to compare $d_t(x,p)$ and $d_T(x,p)$.

By  \cite[Theorem 17.2]{Ham1995} again (see also \cite[Lemma 8.3]{P}), we have
$$d_T(x,p)\geq d_t(x,p)-C_n\sqrt{aT}.$$
If $d_t(x,p) \geq 2C_n\sqrt{aT}$, then $d_T(x,p)\geq \frac{1}{2} d_t(x,p)$. Substitutes it back to \eqref{000}, we have the desired result for $d_t(x,p)\geq 2C_n\sqrt{aT}$. 
If $d_t(x,p)<2C_n\sqrt{aT}$, we can choose a larger constant $C_3$ so that the conclusion holds.
\end{proof}

\subsection{Evolution for energy quantities}
Following the idea in \cite{KB}, we consider the following quantities 
$$h_{ij}=g_{ij}-\tilde g_{ij},\quad A_{ij}^k= \Gamma_{ij}^k-\tilde \Gamma_{ij}^k,\quad \text{and}\;\; S_{ijk}^l=R_{ijk}^l-\tilde R_{ijk}^l.$$
Our goal is to show that all these three quantities vanishes on $[0,T]$.
Now let us recall the evolution equation of $h,A$ and $S$ which can be found in \cite[Page 153-154]{KB}.
\begin{align*}
\partial_t h &= -2S_{ilj}^l,\\
\partial_t A&=g^{-1}*\nabla S+g^{-1}*A*\tilde R+g^{-1}*\tilde g^{-1} *h*\tilde \nabla\tilde R
\end{align*}
and
\begin{align*}
\partial_t S&=\Delta S+\nabla_a (g^{ab}\nabla_b \tilde R-\tilde g^{ab} \tilde\nabla_b \tilde R) + \tilde g^{-1}*A*\tilde \nabla\tilde R\\
&\quad +\tilde g^{-1}*\tilde g^{-1}*h *R*R+\tilde g^{-1}*S*R +\tilde g^{-1} *\tilde R*S.
\end{align*}

Denote $\lambda(x,t)\geq 1$ to be a function on $M\times [0,T]$ such that 
$$\lambda^{-1} \tilde g(t)\leq g(t)\leq \lambda\, \tilde g(t), \;\;\text{on}\;M\times [0,T].$$

For any $a,b>0$, $\bar h=t^{-a}h,\, \bar A=t^{-b} A$ and $|S|_g^2$ satisfies 
\begin{align}
\partial_t (|\bar h|^2)&\leq C_n |R| |\bar h|^2+C_nt^{-a} |S||\bar h|\label{1}\\
\partial_t (|\bar A|^2)&\leq C_n|R||\bar A|^2+C_nt^{-b}|\nabla S||\bar A|+ C_n \lambda^2|\bar A|^2|\tilde R|_{\tilde g}\label{2}\\
&\quad +C_nt^{a-b}\lambda^{4} |\bar A||\bar h||\tilde\nabla\tilde R|_{\tilde g}\notag\\
\partial_t |S|^2&\leq 2\langle \Delta S+ div \,U,S\rangle+C_n\lambda^3 |S|^2 \big( |R|+|\tilde R|_{\tilde g}\big)\label{3}\\
&\quad +C_n \lambda^2t^a |\bar h||S||R|^2+C_n\lambda^4t^b |\bar A||S||\tilde\nabla\tilde R|_{\tilde g}\notag
\end{align}
where $(div \,U)^l_{ijk}=\nabla_a (g^{ab}\nabla_b \tilde R^l_{ijk}-\tilde g^{ab} \tilde\nabla_b \tilde R^l_{ijk})$.

Moreover,
\begin{align}\label{4}
|U|
&\leq C_n\lambda^4 t^b |\bar A||\tilde R|_{\tilde g}+C_n \lambda^4 t^a|\bar h| |\tilde \nabla\tilde R|_{\tilde g}
\end{align}

Here we use $|\cdot |$ and $|\cdot|_{\tilde g}$ to denote the norm with respect to metric $g(t)$ and $\tilde g(t)$ respectively.

In \cite{KB}, the author considered the uniqueness problem where the metrics are uniformly. Equivalently, $\lambda(x,t)$ is assumed to be uniformly bounded on $M\times[0,T]$. In the following, we show that under assumption of Theorem \ref{real}, $\lambda$ can be controlled in term of $d_t (\cdot ,p)$.

\begin{lma}\label{equiv}
Under the assumption of theorem \ref{real}, there exists $C=C(\epsilon,n,m,K,C_0,T)>0$ such that 
$$C^{-1}(d_T(x,p)+1)^{-2m\epsilon} g_0\leq g(t)\leq C(d_T(x,p)+1)^{2m\epsilon} g_0$$
for all $(x,t)\in M\times [0,T]$. In particular, we can pick 
$$\lambda(x,t)= C^2(d_T(x,p)+1)^{4m\epsilon}.$$
\end{lma}
\begin{proof}
Due to Corollary \ref{esti}, when $d_T(x,p) \leq r$ where $r\geq 1$, we have
$$|Ric(g(t))|\leq \left(\frac{\epsilon}{t}\right)^{1-\delta}\left(C_1 r^m\right)^{\delta},\;\;\forall \delta \in [0,1]$$
for some $C_1=C_1(K,n,C_0,m,T)>0$. By integrating it and using the Ricci flow equation, we therefore conclude that for any $\delta\in (0,1]$,
\begin{align}\label{2222}\exp\left[-2\delta^{-1}\epsilon (C_1\epsilon^{-1}r^m)^\delta\right] g_0\leq g(t)\leq \exp\left[2\delta^{-1}\epsilon (C_1\epsilon^{-1}r^m)^\delta\right] g_0.\end{align}
\noindent\\
Now we choose $\delta\in (0,1]$ so that $\delta \cdot \log (C_1\epsilon^{-1} r^m)=\frac{1}{10}$. Noted that $\delta\in (0,1]$ if $r$ is sufficiently large. Then \eqref{2222} becomes
\begin{align}(C_1\epsilon^{-1} r^m)^{-2\epsilon }g_0\leq g(t)\leq (C_1\epsilon^{-1} r^m)^{2\epsilon }g_0.\end{align}
\end{proof}

\subsection{Energy argument}
\begin{proof}[Proof of theorem \ref{real}]
We will modify the estimates so that we can apply the energy argument used in \cite{KB}.

Let $r>>1$ be a sufficiently large constant. By \cite[proposition 2.1]{GW}, there is a smooth function $\rho$ in which $|^{g(T)}\nabla \rho|_{g(T)}\leq 2$ and $|\rho-d_{g(T)}|\leq 1$. Define $\phi(x)=\Phi (\rho(x)/r)$ where $\Phi$ is a smooth non-increasing function defined on $[0,+\infty)$ which is identical to $1$ on $[0,1]$, vanishes outside $[0,2]$ and satisfies $|\Phi'|^2\leq  100\Phi$. Let $\eta$ be the function obtained from Lemma \ref{exh} with $L_i$ to be chosen later. Define the energy quantity to be 
$$\mathcal{E}(x,t)=|\bar h|^2+|\bar A|^2+|S|^2\quad\text{and}\quad E_r(t)=\int_M \mathcal{E}(x,t)e^{-\eta(x,t)}\phi(x)\,d\mu_t.$$
We will specify the choice of $a,b,\epsilon,L_1$ and $L_2$ later.

We would like to point out that in \cite{KB}, the author defined the energy quantity using the same quantity $\mathcal{E}(x,t)$ but with sightly different choice of $a$ and $b$. Moreover, the cutoff function and the estimates in \cite{KB} are in term of $g_0$ while the cutoff function here and estimates in Corollary \ref{esti} are with respect to $g(T)$ and $\tilde g(T)$.

By Corollary \ref{esti}, there exists $C_1=C_1(n,C_0,T,m,K)>0$ such that for any $x\in supp(\phi)$, $t\in (0,T]$, $\delta\in [0,1]$,
\begin{align}\label{5}|R|_g+|\tilde R|_{\tilde g}\leq C_1 \frac{r^{m\delta}}{t^{1-\delta}}\end{align}
and 
\begin{align}\label{6}|\nabla R|_g +|\tilde\nabla \tilde R|_{\tilde g}\leq C_1\frac{r^{3\delta m/2}}{t^{3/2-\delta}}.\end{align}

\noindent On the other hand, by Lemma \ref{equiv}
\begin{align}\label{7}\lambda(x,t) \leq  C_1 r^{4m\epsilon }.\end{align}
For notational convenience, we will denote $N$ to be an generic constant depending only on $n,C_0,T,m,K$ for our convenience. It may vary from line to line.

By substituting \eqref{5}- \eqref{7} into \eqref{1}-\eqref{4} and applying Cauchy inequalities on equations \eqref{1}-\eqref{4}, 
we have on $supp (\phi)\subset B_T(p,3r)$,
\begin{equation}
\begin{split}
\frac{\partial }{\partial t}|\bar h|^2&\leq N\frac{r^{m\delta}}{t^{1-\delta}}|\bar h|^2+\frac{N}{t^a}|S|^2+\frac{N}{t^a} |\bar h|^2,\\
\frac{\partial }{\partial t}|\bar A|^2&\leq N\frac{r^{8m\epsilon+2\delta m}}{t^{\frac{3}{2}-\delta+b-a}}|\bar h|^2+N|\bar A|^2\left[\frac{r^{m\delta}}{t^{1-\delta}}+\frac{1}{t^{2b}}+\frac{r^{m\delta+8m\epsilon}}{t^{1-\delta}}+\frac{r^{8m\epsilon+2\delta m}}{t^{\frac{3}{2}-\delta+b-a}}\right]\\
&\quad+|\nabla S|^2,\\
\frac{\partial}{\partial t} |S|^2&\leq 2\langle \Delta S+div \,U,S\rangle+N\frac{r^{16m\epsilon+2\delta m}}{t^{\frac{3}{2}-\delta-b}}|\bar A|^2+N\frac{r^{8m\epsilon+m\delta}}{t^{1-\delta-a}}|\bar h|^2\\
&\quad +N|S|^2\left[\frac{r^{12m\epsilon+m\delta}}{t^{1-\delta}}+\frac{r^{16m\epsilon+2\delta m}}{t^{\frac{3}{2}-\delta-b}}\right]\\
|U|^2&\leq N\frac{r^{32m\epsilon+2m\delta}}{t^{2-2\delta-2b}}|\bar A|^2+N\frac{r^{32m\epsilon+4\delta m}}{t^{3-2\delta-2a}}|\bar h|^2
\end{split}
\end{equation}
where $\delta$ can be any number in $[0,1]$. Now we specify our choice of $a,b,\delta$ and $\epsilon$. Choose 
\begin{enumerate}
\item[(a)] $\delta=\frac{1}{4m}$, 
\item[(b)] $1>a>1-\delta$, 
\item[(c)] $\frac{1}{2}>b>\frac{1}{2}-\delta$ and 
\item[(d)] $\epsilon=\frac{1}{40m}$.
\end{enumerate} 

Now we are ready to get a differential inequality of $E_r(t)$. The calculation below is similar to the one in \cite[Page 172-174]{KB}. With the above choice of $a,b,\delta$ and $\epsilon$ together with estimate \eqref{5}, we have 
\begin{equation}
\begin{split}\label{EEq}
&\quad \frac{\partial}{\partial t}\int_M e^{-\eta}\phi \left( |\bar h|^2+|\bar A|^2\right)d\mu_t\\
&=\int_Me^{-\eta}\phi \, \partial_t  \left( |\bar h|^2+|\bar A|^2\right) - \frac{\partial\eta}{\partial t} e^{-\eta}\phi \left( |\bar h|^2+|\bar A|^2\right) d\mu_t\\
&\quad - \int_M S_{g}\,e^{-\eta}\phi \left( |\bar h|^2+|\bar A|^2\right)d\mu_t\\
&\leq \frac{Nr^2}{t^\gamma}E_r(t) +\int_M |\nabla S|^2 e^{-\eta}\phi \,d\mu_t
\end{split}
\end{equation}
on $(0,\tau]$ where we have used Lemma \ref{exh}. Similarly, we can also deduce that 
\begin{equation}
\begin{split}\label{EEEQ}
\frac{\partial}{\partial t}\int_M e^{-\eta}\phi  |S|^2d\mu_t
&\leq \frac{Nr^2}{t^\gamma}E_r(t)+\int_M\left(- \frac{\partial\eta}{\partial t}\right)|S|^2 e^{-\eta}\phi\,d\mu_t\\
&\quad +2\int_M e^{-\eta}\phi \langle \Delta S+div U,S\rangle \,d\mu_t\\
&\leq \frac{Nr^2}{t^\gamma}E_r(t)-\int_ML_1|\nabla\eta|^2|S|^2 e^{-\eta}\phi\,d\mu_t\\
&\quad +2\int_M e^{-\eta}\phi \langle \Delta S+div U,S\rangle \,d\mu_t.
\end{split}
\end{equation}
We apply integration by part on the last term yielding 
\begin{equation}
\begin{split}\label{EEEEQ}
&2\int_M\langle  \triangle S+\text{div}\;U,S \rangle\phi e^{-\eta}d\mu\\
&\leq -2\int_M |\nabla S|^2e^{-\eta}\phi d\mu_t\\
&\quad +2\int_M \Big[ |\nabla S||U|\phi+(|\nabla \eta|\phi+|\nabla \phi|)(|S||\nabla S|+|U||S|) \Big]e^{-\eta}d\mu.
\end{split}
\end{equation}
By Cauchy inequality again, we conclude that
\begin{equation}
\begin{split}\label{EEEEEQ}
&\quad 2|\nabla S||U|\phi+2(|\nabla \eta|\phi+|\nabla \phi|)(|S||\nabla S|+|U||S|)\\
&\leq N|U|^2\phi_r+N|\nabla \eta|^2|S|^2\phi_r+N\frac{|\nabla\phi|^2}{\phi}|S|^2+|\nabla S|^2\phi\\
&\leq |\nabla S|^2\phi+N|U|^2\phi+C_2|\nabla \eta|^2|S|^2\phi+N\frac{|\nabla\phi|^2}{\phi}|S|^2,
\end{split}
\end{equation}
for some constant $C_2=C_2(C_1,n,\sigma,\delta,T)$. We choose $L_1=C_2$ so that the term involving $|\nabla\eta|^2$ can be eliminated.

Combining \eqref{EEq}, \eqref{EEEQ}, \eqref{EEEEQ} and \eqref{EEEEEQ}, we arrive at the following inequality 
\begin{align*}
\partial_t E_r&\leq \frac{Nr^2}{t^\gamma}E_r+N\int_{A_{g(T)}(r,2r)}\frac{|\nabla\phi|^2}{\phi}|S|^2 e^{-\eta}d\mu
\end{align*}
on $(0,\tau]$ for some $\tau=\tau(L_2,n,m,C_0,K)>0$. \\

\noindent

In view of \eqref{zerometric}, we have the following estimate.
$$|\nabla \phi|^2= g^{ij}\partial_i \phi\cdot \partial_j \phi\leq \frac{N\phi}{r^2\sqrt{t}}.$$
Together with Corollary \ref{esti} and Lemma \ref{equiv}, the last term is bounded above by
$$N\int_{A_{g(T)}(r,2r)}\frac{|\nabla\phi|^2}{\phi}|S|^2 e^{-\eta}d\mu\leq \frac{N}{\sqrt{t}} e^{-L_2 r^2}V_T(B_T(p,2r))\leq \frac{N}{\sqrt{t}} e^{-L_2 r^2/2}$$
where we have used volume comparison theorem on $g(T)$. Hence, for $r$ sufficiently large
\begin{equation}
\begin{split}
\frac{\partial}{\partial t}E_r(t) \leq \frac{Nr^2}{t^\gamma} E_t +\frac{N}{\sqrt{t}}e^{-L_2r^2/2}
\end{split}
\end{equation}
And thus the function $\bar E_r(t)=\exp\left(-\frac{Nr^2t^{1-\gamma}}{1-\gamma} \right) \cdot E_t(t)$ satisfies
\begin{equation}
\begin{split}
\frac{\partial}{\partial t} \bar E_r(t)& \leq \exp\left(-\frac{Nr^2t^{1-\gamma}}{1-\gamma} -\frac{L_2r^2}{2}\right)\frac{N}{\sqrt{t}} \leq \frac{Ne^{-r^2}}{\sqrt{t}}.
\end{split}
\end{equation}
if we choose $L_2$ to be sufficiently large.
Since  $B_{g(T)}(p,r)$ is relatively compact and hence $B_{g(T)}(p,r)\subset \subset B_{g_0}(p,\bar r)$ for some $\bar r>>1$. By \cite[Lemma 10]{KB}, $\bar E_r(t)\rightarrow 0$ as $t\rightarrow 0^+$. By integrating \eqref{ode} from $s>0$ to $t\in [s,\tau]$, we get 
\begin{align}\label{ode}\bar E_r(t)\leq \bar E_r(s)+2Ne^{-r^2}\sqrt{t}.\end{align}
By letting $s\rightarrow 0^+$ and followed by $r\rightarrow \infty$, we conclude that $h\equiv 0$ on $[0,\tau]$. By the uniqueness theorem \cite{ChenZhu2006} or iterating the arguments, we can then conclude that $g(t)\equiv \tilde g(t)$ on $[0,T]$.

\end{proof}

\end{document}